\documentclass{amsart}
\usepackage[nohug]{diagrams}
\usepackage{graphicx}
\usepackage[mathscr]{eucal}
\usepackage{amssymb, amsmath, amsthm}
\usepackage{amsfonts}
\usepackage{latexsym}
\usepackage{tabularx,array}
\usepackage{enumerate}
\usepackage[all,arc]{xy}
\usepackage{mathrsfs}
\usepackage{latexsym}
\usepackage{color}
\usepackage{mathtools}
\usepackage{leftidx}
\usepackage{bm}
\usepackage{url}
\usepackage{marginnote}
\usepackage[colorinlistoftodos]{todonotes}
\newfont{\msam}{msam10}

\newtheorem{theorem}[]{Theorem}
\newtheorem{proposition}[]{Proposition}
\newtheorem{corollary}[]{Corollary}
\newtheorem{lemma}[]{Lemma}

\theoremstyle{definition}
\newtheorem{definition}[]{Definition}
\newtheorem{remark}[]{Remark}
\newtheorem{defn}[theorem]{Definition}
\newtheorem{example}[]{Example}
\newtheorem{conj}[]{Conjecture}

\let\nc\newcommand

\DeclareMathOperator*{\Motimes}{\text{\raisebox{0.25ex}{\scalebox{0.8}{$\bigotimes$}}}}
\DeclareMathOperator*{\Moplus}{\text{\raisebox{0.25ex}{\scalebox{0.8}{$\bigoplus$}}}}

\def\bthm{\begin{theorem}}
\def\ethm{\end{theorem}}
\def\blemma{\begin{lemma}}
\def\elemma{\end{lemma}}
\def\bproof{\begin{proof}}
\def\eproof{\end{proof}}
\def\bprop{\begin{proposition}}
\def\eprop{\end{proposition}}
\def\bcor{\begin{corollary}}
\def\ecor{\end{corollary}}
\def\bconj{\begin{conj}}
\def\econj{\end{conj}}
\nc{\la}{\label}
\def\O{\mathcal{O}}
\def\Z{\mathbb{Z}}

\def\Q{\mathbb{Q}}

\def\M{\mathcal{M}}

\def\L {\boldsymbol{L}}

\def\bLambda{{\Lambda}}
\def\Top{\mathrm{Top}}

\def\sSet{\mathrm{sSet}}

\def\Mod{\mathrm{Mod}}
\def\sGr{\mathrm{sGr}}
\def\Gr{\mathrm{Gr}}

\def\Com{\mathtt{Com}}

\def\sVect{\mathtt{sVect}}
\def\Alg{\mathrm{Alg}}
\def\sAlg{\mathrm{sAlg}}

\def\DGL{\mathrm{dgLie}}

\def\Comm{\mathrm{Comm}}
\def\sComm{\mathrm{sComm}}
\def\scAlg{\mathrm{sComm}}

\def\DGA{\mathrm{dgAlg}}
\def\DGCA{\mathrm{dgComm}}
\def\cDGA{\mathrm{dgComm}}

\def\D{\mathcal{D}}
\def\C{\mathcal{C}}

\def\M{\mathcal{M}}
\def\cN{\mathcal{N}}

\def\Ho{{\mathtt{Ho}}}
\def\mfa{\mathfrak{a}}
\nc{\CS}{{\tt{CS}}}
\nc{\CR}{{\tt{CR}}}
\nc{\SR}{{\tt{SR}}}
\nc{\ocolim}{{\rm ocolim}}
\nc{\Ob}{{\rm Ob}}
\nc{\Hom}{{\rm{Hom}}}
\nc{\Homcont}{{\mathcal{H}om}}
\nc{\HOM}{\underline{\rm{Hom}}}
\nc{\DER}{\underline{\rm{Der}}}
\nc{\END}{\underline{\rm{End}}}
\nc{\bSym}{\mathbf{Sym}}
\nc{\Ext}{{\rm{Ext}}}
\nc{\Map}{{\rm{Map}}}
\nc{\Rep}{{\rm{Rep}}}
\nc{\DRep}{{\rm{DRep}}}
\nc{\ODRep}{{\mathcal O}{\rm{DRep}}}
\nc{\NCRep}{\widetilde{\rm{Rep}}}
\nc{\RAct}{{\rm{RAct}}}
\nc{\bs}{\backslash}
\nc{\ob}{{\tt{Obs}}}
\nc{\CE}{\mathcal{C}}
\nc{\TP}{{T\!P}}
\nc{\un}{\underline{n}}
\nc{\um}{\underline{m}}
\nc{\rn}{\langle n \rangle}
\nc{\nn}{{{\natural} {\natural}}}
\nc{\n}{{{\natural}}}
\nc{\A}{\mathbb A}
\nc{\B}{{\mathrm{B}}}
\nc{\Ba}{\overline{\mathrm{B}}}
\nc{\bC}{\overline{C}}
\nc{\bOmega}{\boldsymbol{\Omega}}
\nc{\bB}{\boldsymbol{B}}
\nc{\EXT}{\underline{\rm{Ext}}}
\nc{\TOR}{\underline{\rm{Tor}}}
\nc{\hocolim}{\mathrm{hocolim}}
\def\H{\mathrm H}

\def\HC{\mathrm{HC}}
\def\HS{\mathrm{HS}}

\def\HR{\mathrm{HR}}

\nc{\End}{{\rm{End}}}
\nc{\GL}{{\rm{GL}}}
\nc{\gl}{{\mathfrak{gl}}}
\nc{\rgl}{\overline{{\mathfrak{gl}}}}
\nc{\g}{{\mathfrak{g}}}
\nc{\h}{{\mathfrak{h}}}
\nc{\PGL}{{\rm{PGL}}}
\nc{\SL}{{\rm{SL}}}
\nc{\sll}{\mathfrak{sl}}
\nc{\cn}{ \mbox{\rm c\^{o}ne} }
\nc{\PSL}{{\rm{PSL}}}
\nc{\ad}{{\rm{ad}}}
\nc{\Ad}{{\rm{Ad}}}
\nc{\dlim}{\varinjlim}
\nc{\plim}{\varprojlim}
\nc{\colim}{{{\rm colim}}}

\newcommand{\HH}{{\rm{HH}}}

\newcommand{\Tor}{{\rm{Tor}}}

\newcommand{\Sym}{S}

\newcommand{\Tr}{{\rm{Tr}}}

\newcommand{\into}{\,\hookrightarrow\,}

\newcommand{\sonto}{\,\stackrel{\sim}{\twoheadrightarrow}\,}

\def\cb{\Omega}

\def\bs{\backslash}

\def\ffgr{\mathfrak{G}}

\def\lgr{\mathbb{G}}

\def\cL{\mathcal{L}}

\def\sLie{\mathtt{sLie}}

\nc{\env}{\mathrm{End}(V)}
\nc{\cC}{\mathcal{C}}
\numberwithin{equation}{section}
\numberwithin{theorem}{section}
\numberwithin{lemma}{section}
\numberwithin{proposition}{section}
\numberwithin{definition}{section}
\numberwithin{corollary}{section}
\numberwithin{example}{section}
\numberwithin{remark}{section}

\def\cC{\mathcal{C}}

\newcommand{\sym}{\operatorname{sym}}

\newcommand{\rH}{\overline{\mathrm{H}}}
\def\cO{\mathcal O}

\def\fM{\mathfrak{M}}
\newcommand{\sMod}{\mathrm{sMod}}

\def\bdf{\begin{defn}}
\def\edf{\end{defn}}

\def\brm{\begin{remark}}
\def\erm{\end{remark}}

\theoremstyle{definition}
\def\bdf{\begin{definition}}
\def\edf{\end{definition}}

\newcommand{\SP}{\mathrm{SP}}

\newcommand{\bS}{{\mathbb S}}

\newcommand{\bG}{{\mathbb G}}
\newcommand{\bF}{{\mathbb F}}

\newcommand{\cA}{{\mathcal A}}

\newcommand{\cH}{{\mathcal H}}

\def\arbreBA{\vcenter{\xymatrix@R=2pt@C=2pt{
&&&&\\
&&&*{}\ar@{-}[ul] & \\
&&*{}\ar@{-}[uurr] \ar@{-}[uull] \ar@{-}[d]     &&\\
&&&&
}}}

\def\arbreAB{\vcenter{\xymatrix@R=2pt@C=2pt{
&&&&\\
&*{}\ar@{-}[ur] &&& \\
&&*{}\ar@{-}[uurr] \ar@{-}[uull] \ar@{-}[d]     &&\\
&&&&
}}}

\def\arbreABC{\vcenter{\xymatrix@R=1pt@C=1pt{
&&&&&&\\
&*{}\ar@{-}[ur] &&&&& \\
&&*{}\ar@{-}[uurr] &&&&\\
&&&*{}\ar@{-}[uuurrr] \ar@{-}[uuulll] \ar@{-}[d] &&&\\
&&&&&&
}}}

\def\arbreBAC{\vcenter{\xymatrix@R=1pt@C=1pt{
&&&&&&\\
&&&*{}\ar@{-}[ul] &&& \\
&&*{}\ar@{-}[uurr] &&&&\\
&&&*{}\ar@{-}[uuurrr] \ar@{-}[uuulll] \ar@{-}[d] &&&\\
&&&&&&
}}}

\def\arbreACB{\vcenter{\xymatrix@R=1pt@C=1pt{
&&&&&&\\
&*{}\ar@{-}[ur] &&&&& \\
&&&&*{}\ar@{-}[uull] &&\\
&&&*{}\ar@{-}[uuurrr] \ar@{-}[uuulll] \ar@{-}[d] &&&\\
&&&&&&
}}}

\def\arbreBCA{\vcenter{\xymatrix@R=1pt@C=1pt{
&&&&&&\\
&&&&&*{}\ar@{-}[ul] & \\
&&*{}\ar@{-}[uurr] &&&&\\
&&&*{}\ar@{-}[uuurrr] \ar@{-}[uuulll] \ar@{-}[d] &&&\\
&&&&&&
}}}

\def\arbreCAB{\vcenter{\xymatrix@R=1pt@C=1pt{
&&&&&&\\
&&&*{}\ar@{-}[ur] &&& \\
&&&&*{}\ar@{-}[uull] &&\\
&&&*{}\ar@{-}[uuurrr] \ar@{-}[uuulll] \ar@{-}[d] &&&\\
&&&&&&
}}}

\def\arbreCBA{\vcenter{\xymatrix@R=1pt@C=1pt{
&&&&&&\\
&&&&&*{}\ar@{-}[ul] & \\
&&&&*{}\ar@{-}[uull] &&\\
&&&*{}\ar@{-}[uuurrr] \ar@{-}[uuulll] \ar@{-}[d] &&&\\
&&&&&&
}}}

\def\arbreACA{\vcenter{\xymatrix@R=1pt@C=1pt{
&&&&&&\\
&*{}\ar@{-}[ur] &&&&*{}\ar@{-}[ul] & \\
&&&&&&\\
&&&*{}\ar@{-}[uuurrr] \ar@{-}[uuulll] \ar@{-}[d] &&&\\
&&&&&&
}}}

\begin{document}

\title{Symmetric homology is representation homology}
\author{Yuri Berest}
\address{Department of Mathematics,
Cornell University, Ithaca, NY 14853-4201, USA}
\email{berest@math.cornell.edu}
\author{Ajay C. Ramadoss}
\address{Department of Mathematics,
Indiana University,
Bloomington, IN 47405, USA}
\email{ajcramad@indiana.edu}
\begin{abstract}
Symmetric homology is a natural generalization of cyclic homology, in which symmetric groups play the role of cyclic groups. In the case of associative algebras, the symmetric homology theory was introduced by Z. Fiedorowicz \cite{F} and was further developed in the work of S. Ault \cite{Au1, Au2}. In this paper, we show that, for algebras defined over a field of characteristic $0$, the symmetric homology theory is naturally equivalent to the (one-dimensional) representation homology theory introduced by the authors (jointly with G. Khachatryan) in  \cite{BKR}. Using known results on representation homology, we compute symmetric homology explicitly for basic algebras, such as polynomial algebras and universal enveloping algebras of (DG) Lie algebras. As an application, we prove two conjectures of Ault and Fiedorowicz, including the main conjecture of \cite{AF07} on topological interpretation
of symmetric homology of polynomial algebras.
\end{abstract}
\maketitle
\maketitle


\vspace*{-1ex}

\section{Introduction and Main Results} \la{S1}
Cyclic homology was introduced by A. Connes \cite{C83} (and independently by B. Tsygan \cite{T83}) almost 40 years ago. In Connes' approach, the key idea was to extend the classical notion of a simplicial module to that of a {\it cyclic module}, which he defined as a functor $ \Delta C^{\rm op} \to \Mod_k $ on the so-called cyclic category $ \Delta C $. Connes observed that, for any associative unital $k$-algebra $A$, the standard simplicial module $ B^{\rm Hoch}(A): \Delta^{\rm op} \to \Mod_k $, computing the Hochschild homology of $A$, extends naturally to a cyclic module $ B^{\rm cyc}(A): \Delta C^{\rm op} \to \Mod_k $ called the {\it cyclic bar construction} of $A$. By analogy with the well-known `Tor-formula' for  Hochschild homology:
\begin{equation}
\la{HHA} 
\HH_*(A) \,=\, \Tor^{\Delta^{\rm op}}_{\ast}(k,\,B^{\rm Hoch}A)
\end{equation}
he then defined the {\it cyclic homology of $A$} by 
\begin{equation}
\la{HCA} 
\HC_*(A) \,:=\, \Tor^{\Delta C^{\rm op}}_{\ast}(k,\,B^{\rm cyc}A)\,.
\end{equation}
Now, Connes' cyclic category $ \Delta C $ is an extension of the simplicial category
$ \Delta $ obtained by taking a ``crossed product''  with the family
(groupoid) of cyclic groups $ \{C_n = \Z/(n+1)\Z\}_{n \ge 0}\,$. This idea of a categorical
crossed product was formalized by Loday and Fiedorowicz \cite{FL91} who introduced the notion of a crossed simplicial group category $ \Delta G $ and constructed new examples of such categories for various families of groups $ \{G_n\}_{n \ge 0}$. Apart from $\Delta$ and $ \Delta C$, the most interesting example\footnote{This and other examples of \cite{FL91} also appeared in the work of Feigin-Tsygan \cite{FT87} and Krasauskas \cite{Kr87}.} is the {\it symmetric category} $ \Delta S $ associated to the family of symmetric groups $ \{\Sigma_{n+1}\}_{n \ge 0} $. In the unpublished paper \cite{F}, Fiedorowicz showed that, for any algebra $A$, there is a natural functor
$ B_{\rm sym}(A): \Delta S \to \Mod_k $, called
the {\it symmetric bar construction}\footnote{Note that, unlike Hochschild and cyclic, the symmetric bar construction is defined as a {\it covariant} functor
on the category $ \Delta S$, which we indicate by placing `sym' as a subscript in $B_{\rm sym}$. We recall the precise definition of $ \Delta S $ and the construction of  $B_{\rm sym}(A)\,$ in Section~\ref{S2.3} below.}, and, by analogy with \eqref{HHA} and \eqref{HCA}, he then defined the {\it symmetric homology of $A$} by the formula
\begin{equation}
\la{HSA} 
\HS_*(A) \,:=\, \Tor^{\Delta S}_{\ast}(k,\,B_{\rm sym} A)\,.
\end{equation}
It turns out that, despite its resemblance to \eqref{HHA} and \eqref{HCA}, formula \eqref{HSA} yields a much more complicated homology theory. While the Hochschild and cyclic homology are explicitly known and can be easily calculated for many different types of algebras (see \cite{L}), the symmetric homology $ \HS_*(A) $ is not even known for the polynomial algebras $A = k[x_1,\ldots,x_n]$ in $ n\ge 2 $ variables!
The main problem seems to be that the $ \HS_*$ theory is less accessible to algebraic methods than that of $ \HH_* $ and $\HC_*$, and --- as was already recognized in \cite{F} --- has to be approached through topology. 

In 2007, Ault and Fiedorowicz \cite{AF07} announced a number of general theorems on the structure of $\HS_*(A)$ and an explicit conjecture on polynomial algebras  (see Conjecture~\ref{Con1} below). The proofs of most of the theorems of \cite{AF07} (along with results of the original preprint \cite{F}) appeared in Ault's Ph.D. dissertation \cite{Au1} and his subsequent paper \cite{Au2}. However, the main conjecture of \cite{AF07} seems still to be open:
\begin{conj}[\cite{AF07}, Conjecture 1] 
\la{Con1}
For any $ n \ge 1 $, there is an isomorphism of graded algebras
\begin{equation}\la{Eq1}
{\rm HS}_\ast(k[x_1,\ldots,x_n]) \,\cong\, \H_\ast\bigg(\prod_{i=1}^n \cC_{\infty}(\bS^0) \times \prod_{i=2}^n \, \Omega^{\infty}\Sigma^{\infty}(\bS^{i-1})^{\binom{n}{i}}\, ,\, k\bigg) \ ,
\end{equation}
where $ \Omega^{\infty}\Sigma^{\infty} $ is the stable homotopy functor and
$\cC_\infty $ is the monad associated to the little $\infty$-cubes operad, both defined
as functors on the category of based topological spaces $ \Top_{\ast} $ (see, e.g., \cite{May72}).
\end{conj}
Another conjecture of \cite{AF07} that apparently remains unproven\footnote{
Conjecture~\ref{Con2} is stated in \cite{AF07} as Theorem~8. However, 
as remarked in \cite{Au2}, this theorem remains unproven; in fact, the author of \cite{Au2} suspects that it is false in general (see {\it loc. cit.}, Remark 3.3).
We will show, perhaps unsurprisingly, that Conjecture~\ref{Con2} does hold if 
$k$ is a field of characteristic $0$ (see Corollary~\ref{bdta}).} provides a topological formula for the symmetric homology of an arbitrary algebra:
\begin{conj}[\cite{AF07}] 
\la{Con2}
For any $k$-algebra $A$, there is an isomorphism
\begin{equation}\la{Eq2}
{\rm HS}_\ast(A) \,\cong \,\H_\ast(B(D,T,A),\,k)\ ,
\end{equation}
where $B(D,T,A)$ stands for the two-sided bar construction associated to the monad $D$ of 
an $E_{\infty}$-operad in the category of simplicial $k$-modules corresponding to the classical Barratt-Eccles operad \cite{BE74} and $T$ is the tensor algebra functor on $k$-modules.
\end{conj}

The goal of the present paper is to clarify the situation with symmetric homology ---
in particular, to prove Conjecture~\ref{Con1} and Conjecture~\ref{Con2} stated above ---
in the rational case: i.e., for algebras $A$ defined over a filed $k$ of characteristic $0$.
Our approach is based on Theorem~\ref{T1} that provides a natural identification of symmetric homology with (one-dimensional) {\it representation homology}\, 
defined in \cite{BKR} and studied in our subsequent papers (see, e.g., \cite{BR, BFPRW, BRYI, BRYII, BR22}). Representation homology originates from derived algebraic geometry; however, it can be defined algebraically, in a straightforward way, using only classical homotopical algebra \cite{Q1}. The advantage of this algebraic approach is that it makes the relation to cyclic homology quite natural. In the rest of the Introduction, we briefly review the definition of
representation homology and its relation to cyclic homology, following \cite{BKR}, and then state our main theorems. Unlike in \cite{BKR}, however, we will work with simplicial (rather than DG) algebras over an arbitrary commutative ring.

Let $ \Alg_k $ and $ \Comm_k $ denote the categories of associative and commutative algebras (with $1$), both defined over a fixed commutative ring $k$. 
For an integer $n \geqslant 1$, consider the matrix functor $\,M_n:\, \Comm_k \to \Alg_k\,$ that takes a commutative $k$-algebra $C$ to the $k$-algebra $M_n(C)$ of $ (n\times n)$-matrices with entries in $C$. It is classically known and easy to prove (see, e.g., \cite{Co79} and \cite[Prop. 2.1]{BKR} for a general result) that $M_n$ has a left adjoint
\begin{equation}
\la{adj1}
(\,\mbox{--}\,)_n:\ \Alg_k \to \Comm_k
\end{equation}
called the {\it $n$-dimensional representation functor} on $ \Alg_k $. This terminology is motivated by the fact that for any algebra $A$, $\,(A)_n = \O[\Rep_n(A)] \,$ is the coordinate ring of the affine
moduli scheme $ \Rep_n(A) $ parametrizing the $n$-dimensional $k$-linear representations of $A$. Now, to define representation homology we simply `derive' the functor \eqref{adj1},
following the standard procedure in homotopical algebra: we prolong the adjunction $\,(\,\mbox{--}\,)_n: \Alg_k \rightleftarrows \Comm_k: M_n\,$ to simplicial algebras, using degreewise extension:
\begin{equation}
\la{adj2}
(\,\mbox{--}\,)_n:\ \sAlg_k \,\rightleftarrows\, \sComm_k\ : M_n
\end{equation}
and then replace the adjoint functors \eqref{adj2} with their universal homotopy approximations represented by derived functors. The key point here
(observed in \cite{BKR} in the case of DG algebras) is that \eqref{adj2} is a Quillen adjunction with respect to the standard projective model structures on simplicial 
algebras\footnote{This follows from the fact that the underlying simplicial set of $M_n(C)$ for any $ C \in \sComm_k $ is simply the product of $n^2$ copies of $C$, and hence the matrix functor preserves both weak equivalences and fibrations, i.e. it is a right
Quillen functor on simplicial algebras (see, e.g., \cite{GS07}).}. Hence the derived functors of \eqref{adj2} exist and are well behaved: in particular, by Quillen's Adjunction Theorem, they form an adjoint pair at the level homotopy categories:
$$
\L(\,\mbox{--}\,)_n:\ \Ho(\sAlg_k) \,\rightleftarrows\, \Ho(\sComm_k)\, : M_n\ . 
$$
For a (simplicial) $k$-algebra $ A $, we can now define its {\it $n$-dimensional representation homology} by
\begin{equation}
\la{RHDefAlg} 
\HR_*(A, k^n)\, := \, \L_\ast(A)_n\,=\, \pi_\ast[(QA)_n]\,,
\end{equation}
where $\, QA \sonto A\,$ is a cofibrant simplicial resolution
of $A$. By definition, $ \HR_\ast(A, k^n) $ is a graded commutative $k$-algebra, with $\, \HR_0(A, k^n) \cong (A)_n\,$, that does not depend on the choice of the
resolution $QA$. For $ n = 1 $, the matrix functor is just the inclusion $ \sComm_k \into \sAlg_k $, and hence its left adjoint --- the $1$-dimensional derived representation functor $  \L(\,\mbox{--}\,)_1: \,\Ho(\sAlg_k) \to \Ho(\sComm_k) $ --- coincides with the derived abelianization (`commutativization') of algebras.

Now, to relate representation homology to cyclic homology we recall that 
$\, \HC_0(A) = A/[A,A]\,$ for any $ A \in \Alg_k$, and observe that there is an obvious trace map $\,\Tr_n(A):\ \HC_0(A)\,\to\,(A)_n\,$
%
%
induced by the character of the universal $n$-dimensional representation of $A$. It turns out that this classical character map extends naturally to higher cyclic homology taking values in representation homology: i.e., for any $n\ge 1$, there is a canonical graded
$k$-linear map
\begin{equation}
\la{DTr}
\Tr_n(A)_\ast:\ \HC_\ast(A)\,\to\,\HR_\ast(A, k^n)
\end{equation}
called the {\it derived character map of $n$-dimensional representations of $A$}. 
The map \eqref{DTr} was originally constructed in \cite{BKR} for associative DG algebras
over a field $k$ of characteristic zero. However, working simplicially, we can construct
\eqref{DTr} for $A$ defined over any field (or even an arbitrary commutative ring $k$, provided $A$ is flat over $k$; see Section~\ref{S3.1}). For $n=1$, 
to simplify the notation we denote \eqref{DTr} as
\begin{equation}
\la{Tr*}
\Tr(A)_\ast:\ \HC_\ast(A)\,\to\,\HR_\ast(A, k)
\end{equation}

Next, recall that, for any algebra $A$, there is also a natural map
\begin{equation}
 \la{HCS}
 i^*:\ \HC_\ast(A)\,\to\,\HS_*(A)
\end{equation}
relating cyclic homology to symmetric homology via the inclusion of categories $ i: \Delta C^{\rm op} \into \Delta S $. (The existence of this map has been noticed already in \cite{F} (see, {\it loc. cit.}, Remark 1.4), and it is studied in \cite[Section 12]{Au1}; we will review its construction in Section~\ref{S3.1} below.) Our first observation (see Proposition~\ref{factorTr}) is that the derived character map \eqref{Tr*} factors 
through \eqref{HCS}, inducing a natural linear map that relates symmetric homology to
representation homology:
\begin{equation}
\la{Triso}
\SR(A)_*:\ \HS_\ast(A)\, \to\, \HR_*(A, k) 
\end{equation}
Finally, to state our main theorem we recall a result of Ault (see \cite[Theorem 4.1]{Au2}) showing that the $ \HS_*(A) $ for any $A$ admits natural 
homology operations and a Pontryagin product making it a graded commutative $k$-algebra. 
\begin{theorem}
\la{T1}
For any associative unital algebra $A$ defined over a field $k$ of characteristic $0$,  the map \eqref{Triso} is an isomorphism of graded commutative algebras: 
\begin{equation}
\la{Iso}
\HS_\ast(A)\, \cong\, \HR_*(A, k) 
\end{equation}
\end{theorem}
%


%
%
\begin{remark}
\la{Rem1}
One may wonder whether the isomorphism \eqref{Iso} holds in general, i.e. for algebras defined over a field $k$ of positive characteristic or even an arbitrary commutative ring. This is not the case: the simplest counterexample is $ A= k[x] $, the polynomial ring of one variable. Since $ k[x] $ is free as a $k$-algebra over any commutative ring $k$, we have $\, \HR_i(k[x], k^n) = 0 \,$ for all $ i>0 $ and all $n \ge 1$. On the other hand, Ault's (computer-assisted) calculations (see \cite{Au1}, Section 11.4 and, in particular, Conjecture 94) show that $ \HS_1(k[x]) \not=0 $ if $ k = \bF_2$ or $ k = \Z \,$. 
\end{remark}
\begin{remark}
\la{Rem2}
For $ A = k[\Gamma] $, the group algebra of a (simplicial) group
$\Gamma $, the map \eqref{Triso} was constructed in our previous paper \cite{BR22}. In this case, \eqref{Triso} has a topological origin: it is induced (on homology) by a natural map of spaces:
\begin{equation}
\SR_{B\Gamma}:\ \Omega\, \Omega^{\infty} \Sigma^{\infty}\!(B \Gamma) \,\to\,   \Omega\,\SP^{\infty}\!(B\Gamma)\,,
\end{equation}
where $\,\SP^\infty\! (B\Gamma)\,$ denotes the classical Dold-Thom space of the classifying 
space of $ \Gamma\,$. The isomorphism \eqref{Iso} was established in \cite{BR22} for
$ A = k[\Gamma] $ under the additional assumption that $ B\Gamma $ has homotopy type of a simply connected CW complex of (locally) finite type over $\Q$ (see \cite[Corollary 5.1]{BR22}). We strengthen this last result of \cite{BR22}
by removing the above additional assumption (see Corollary~\ref{GrHS}).
\end{remark}
Despite its modest appearance Theorem~\ref{T1} is a useful result as it allows one 
to translate known facts about representation homology to symmetric homology. First of all, we can prove (see Section~\ref{S3.3}):
\begin{theorem}
\la{TPoly}
Let $V$ be a finite-dimensional vector space over a field $k$ of 
characteristic $0$, and let $ S_k(V) $ denote the symmetric algebra of $V$.
There is a natural isomorphism of graded commutative algebras
\begin{equation*}
\la{Eq11a}
{\rm HS}_\ast[S_k(V)] \,\cong\, \Lambda_k\big( \Moplus_{i=1}^{\dim V} \wedge^i V [i-1] \big)\,, 
\end{equation*}
where $ \Lambda_k $ stands for the graded symmetric algebra over $k$.
\end{theorem}
In fact, Theorem~\ref{TPoly} is a trivial case of the following general
result on symmetric homology of universal enveloping algebras that we will
prove using results of \cite{BFPRW} on representation homology of Lie algebras
(see Theorem~\ref{hsua}):
\begin{theorem}
\la{T2}
Let $ U \mfa $ be the universal enveloping algebra of a Lie algebra $\mfa$ defined
over a field $k$ of characteristic $0$.  There is a natural isomorphism of graded commutative algebras
$$
\HS_\ast(U \mfa)\,\cong\,\Lambda_k[\,\rH_{\ast+1}(\mfa,\,k)\,]\,,
$$
where $ \rH_\ast(\mfa,\,k) $ is the $($reduced$)$
Chevalley-Eilenberg Lie algebra homology of $ \mfa $ with trivial coefficients.
\end{theorem}

\noindent
As a consequence of Theorem~\ref{TPoly}, it is not hard to deduce (modulo well-known topological results) that 
\begin{corollary}[see Corollary~\ref{hspoly}]
Conjecture~\ref{Con1} holds true over a field $k$ of characteristic $0$.
\end{corollary}
\noindent
In a similar fashion, using Theorem~\ref{T1} and results of \cite{Au1}, we will show that
%
\begin{corollary}[see Corollary~\ref{bdta}]
Conjecture~\ref{Con2} holds true over a field $k$ of characteristic $0$.
\end{corollary}

The paper is organized as follows. In Section~\ref{S2}, we review definitions and
prove some results on representation homology needed for the present paper. Of independent interest here may be Theorem~\ref{grouptoalg} that compares two different kinds of representation homology: for groups and associated group algebras. 
The proofs of our main results appear in Section~\ref{S3}. In Section~\ref{S3.1},
we construct the derived character map \eqref{DTr}  
and show that, when $n=1$ and $k$ is a field, it factors through symmetric homology (Proposition~\ref{factorTr}). In Section~\ref{S3.2}, we prove Theorem~\ref{T1}
in a slightly more general form: for any simlicial $k$-algebra (Theorem~\ref{T1a}).
The key fact behind this result is Proposition~\ref{pC}, for which we decided to
provide two independent proofs, both relying on (different) topological arguments.
In Section~\ref{S3.2}, we also verify Conjecture~\ref{Con2} (Corollary~\ref{bdta}) and give another version of our main theorem: for associative DG algebras
(Proposition~\ref{tA}). Finally, in Section~\ref{S3.3}, we prove Theorem~\ref{T2}
(in a more general setting of chain DG Lie algebras) and deduce from it Theorem~\ref{TPoly} and Conjecture~\ref{Con1} (Corollary~\ref{hspoly}) stated above.

\subsection*{Notation and conventions} Throughout the paper, $ k $ denotes a commutative ground ring. 
Beginning with Section~\ref{S3.2}, we will specialize $k$ to be a field of characteristic zero. 
$\Mod_k$ denotes the category of $k$-modules; an unadorned tensor product $\, \otimes \,$ stands for the tensor product $\, \otimes_k \,$ over $k$. 
An algebra means an associative $k$-algebra with $1$; the category of such algebras is denoted $ \Alg_k $. A commutative algebra means a commutative associative $k$-algebra with $1$; the category of such algebras is denoted $ \Comm_k $. The corresponding categories of simplicial and differential-graded algebras are denoted $ \sAlg_k$, $ \sComm_k$, $\DGA_k$,  and $ \DGCA_k$, respectively.
If $k$ is a field and $V$ is a graded vector space over $k$, we write $ \Lambda_k(V) $ for its graded
symmetric algebra  over $k\,$: thus $\, \Lambda_k(V) := S_k(V_{\rm ev}) \otimes E_k(V_{\rm odd}) $, where $ S_k(V_{\rm ev}) $ and $ E_k(V_{\rm odd}) $ are the symmetric and exterior
algebras of the even and odd subspaces of $V$, respectively.

\section{Preliminaries} 
\la{S2}
\subsection{Functor tensor products}
\la{S2.1}
We begin by recalling the classical construction of functor tensor products (see \cite[Appendix C.10]{L}). For a  small category $\cC$, we let ${\rm Mod}_k(\cC)$ and ${\rm Mod}_k(\cC^{\rm op})$ denote the categories of all covariant and contravariant functors from $\cC$ to ${\rm Mod}_k$, respectively. It is well known that these are abelian categories with sufficiently many  projective and injective objects. Moreover, there is a natural bi-additive functor
\begin{equation*} 
\la{s2e1} \mbox{--} \otimes_{\cC} \mbox{--}\,:\, {\rm Mod}_k(\cC^{\rm op}) \times {\rm Mod}_k(\cC) \to {\rm Mod}_k 
\end{equation*}
called the {\it functor tensor product}.
Explicitly, for  $ \M: \cC \to \Mod_k $ and $ \cN:\cC^{\rm op}\to \Mod_k $, it is defined by
\begin{equation} 
\la{s2e2}
\cN \otimes_{\cC} \M \, := \, \big[\Moplus_{c \in \cC} \cN(c) \otimes \M(c) \big]/R
\end{equation}
where $R$ is the $k$-submodule generated by elements of the form $ \cN(\varphi)(x) \otimes y -x \otimes \M(\varphi)(y)$ for all $x \in \cN(c') $, $y \in \M(c)$ and $\varphi \in {\rm Hom}_{\cC}(c,c')$. We extend the bifunctor \eqref{s2e1} to the categories of
chain complexes of $\cC$-modules (which we denote by ${\rm Ch}({\rm Mod}_k(\cC^{\rm op}))$ and ${\rm Ch}({\rm Mod}_k(\cC))$) and define 
\begin{equation} \la{s2e3} {\rm Tor}^{\cC}_\ast(\cN,\M)\,:=\, {\rm H}_{\ast}(\cN \otimes^{\L}_{\cC} \M) \end{equation}
for $\cN \in {\rm Ch}({\rm Mod}_k(\cC^{\rm op}))$ and $\M \in {\rm Ch}({\rm Mod}_k(\cC))$. Note that $\cN \otimes^{\L}_{\cC} \M$ is an object in the (unbounded) derived category $\mathscr{D}(k) := \mathscr{D}({\rm Mod}_k)$ of $k$-modules, and \eqref{s2e3} is just the usual hyper-Tor functor on chain complexes. Next, we observe that there is a natural functor 
transforming the $\cC^{\rm op}$-diagrams in $ \sSet $ (simplicial presheaves on $\cC$) to 
chain complexes over ${\rm Mod}_k(\cC^{\rm op})$:
\begin{equation} \la{s2e4} \sSet^{\cC^{\rm op}} \,\xrightarrow{\,k[\,\mbox{--}\,]\,} \, 
{\rm sMod}_k({\cC}^{\rm op}) \xrightarrow{N} \, {\rm Ch}({\rm Mod}_k(\cC^{\rm op}))\,,
\end{equation}
where $k[\,\mbox{--}\,]$ is the ($\C$-objectwise extension of the) free module functor on
simplicial sets and $N$ is the classical Dold-Kan normalization functor. Abusing notation, we will write \eqref{s2e4} simply as $\,k[\,\mbox{--}\,]\,$. 

Finally, we recall one well-known result on derived functor tensor products that we will use repeatedly in this paper. Let $ p: \cC \to \mathcal{D} $ be a map of small categories. The restriction functor $ p^*: \Mod_k(\D) \to \Mod_k(\C) $ associated to $p$ is exact and has a left adjoint $\, p_!: \Mod_k(C) \to
\Mod_k(\D) \,$, which is a right exact functor on chain complexes. Replacing $ p_! $ with its
(abelian) left derived functor $ \L p_!$, we get a natural adjunction between derived categories:
$$ \L p_! \,:\, \mathscr{D}({\rm Mod}_k(\cC)) \rightleftarrows \mathscr{D}({\rm Mod}_k(\mathcal{D}))\,:\, p^{\ast}\ .$$
\begin{lemma} \la{torpsh}
For $\cN \in {\rm Mod}_k(\mathcal{D}^{\rm op})$, $\M \in {\rm Mod}_k(\cC)$ there is a natural isomorphism in $\mathscr{D}(k)$:
$$ p^{\ast}\cN \otimes^{\L}_{\cC} \M\,\cong\, \cN \otimes^{\L}_{\mathcal{D}} \L p_!\M\ . $$
\end{lemma}
\noindent
The proof of Lemma~\ref{torpsh} and some of its consequences can be found, for example, in {\cite[Prop. 1.1]{Dja19}}.

\subsection{Representation homology} 
\la{S2.2}
We review the definition of representation homology of (simplicial) groups and monoids, following the approach 
of \cite{BRYI}.
Let $\ffgr$ denote the small category whose objects $\langle n \rangle$ are the finitely generated free groups $\mathbb{F}_n = \bF\langle x_1, x_2, \ldots, x_n \rangle$, one for each $ n \ge 0 $ (with convention that $ \langle 0 \rangle $ is the trivial group), and the morphisms are arbitrary group homomorphisms. Every simplicial group $\Gamma \in \sGr$ (i.e. a simplicial object in $\Gr$) defines a functor 
$ \underline{\Gamma}: \ffgr^{\rm op} \to \sSet\,$,$\,\langle n \rangle \mapsto \Gamma^n $,
where $\Gamma^n$ denotes the product of $n$ copies of (the underlying simplicial set of) $\Gamma$ in $\sSet$. Composing this functor with $k[\mbox{--}]$ gives the simplicial $ \ffgr^{\rm op}$-module 
\begin{equation} \la{s2e6} \underline{k[\Gamma]}: \ffgr^{\rm op} \to \sMod_k\,,\qquad \langle n \rangle \mapsto k[\Gamma^n] \cong k[\Gamma]^{\otimes n} \ . \end{equation}
More generally, for any simplicial (or DG) cocommutative Hopf algebra $\cA$, we can define
\begin{equation} \la{rgmod} \underline{\cA}: \ffgr^{\rm op} \to {\rm Ch}({\rm Mod}_k) \,,\qquad \langle n \rangle \mapsto \cA^{\otimes n}\ .\end{equation}
where $\sMod_k$ is identified with ${\rm Ch}(\Mod_k)$ via the Dold-Kan normalization functor.
Dually (see, e.g., \cite[Prop. 14.1.6]{R20}), every {\it commutative} Hopf algebra algebra $ \cH $ defines a covariant functor (left $\ffgr$-module) by the rule
\begin{equation} 
\la{halg}
\underline{\cH}:\ffgr \to {\rm Mod}_k\ ,\quad  \langle n \rangle \mapsto \cH^{\otimes n}\ .
\end{equation}
For example, if $G$ is an affine algebraic group (e.g., $G={\rm GL}_n(k)$) with coordinate ring $\cH = \cO(G)$, then \eqref{halg} can be written in the form $\langle n \rangle \mapsto \cO[{\rm Rep}_G(\langle n \rangle)]$ which makes the functoriality clear. 

\begin{definition} \la{RHDef}
The {\it representation homology} of a group $\Gamma \in \sGr$ with coefficients in $\cH$ is defined by 
\begin{equation*}
\la{HRGr}
{\rm HR}_{\ast}(\Gamma, \cH)\,:=\,{\rm Tor}^{\ffgr}_{\ast}(\underline{k[\Gamma]},\,\underline{\cH}) \ .
\end{equation*}
In the special case when $G$ is an affine algebraic group over $k$ and $ \cH = \cO(G) $, we write
${\rm HR}_{\ast}(\Gamma, G) $ instead of ${\rm HR}_{\ast}(\Gamma, \cO(G))\,$ to simplify the notation.
\end{definition}
Next, recall that there are classical adjoint functors (originally constructed in \cite{Kan1}) that relate the (model) category of simplicial groups to that of (reduced) simplicial sets:
\begin{equation} \la{klg} 
\lgr\,:\, \sSet_0 \rightleftarrows \sGr\,:\,\bar{W}
\end{equation}

The left adjoint $\lgr$ is called the {\it Kan loop group functor}, and the right adjoint $\bar{W}$ is the {\it classifying complex functor} on simplicial groups.  The properties of these functors are well known and described, for example, in \cite[Chap. V]{GJ} (see also \cite[Sect. 2.2]{BRYI}). We mention only that the pair \eqref{klg} is a Quillen equivalence, both $\lgr$ and $\bar{W}$ being homotopy functors (see \cite[V.6.4]{GJ}). By \cite[Lemma 3.2]{BR22}, $\HR_\ast(\Gamma,\cH)$ depends only on the homotopy type of the classifying space $\bar{W}\Gamma$.   

Representation homology can be defined naturally --- in a way similar to Definition~\ref{RHDef} --- for other types of algebraic objects (more precisely, simplicial algebras over arbitrary algebraic theories, see \cite{BR22b}). Here, for purposes of comparison, we will give such a definition for the theory of monoids.

Let $\fM \subset {\rm Mon} $ denote the category whose objects are finitely generated free monoids\footnote{Abusing notation, we will use the same symbols to denote the objects of $\fM$ and $\ffgr$.} $ \langle n \rangle $, one for each $ n \ge 0 $. Group completion yields a natural functor (map of algebraic theories)$\,p:\fM \to \ffgr\,$ that takes the free monoid on $n$ generators to the free group on $n$ generators (for $n \geqslant 0$). Now, for any fixed integer $n \geqslant 1$, the affine monoid $k$-scheme $M_n$ (of $n \times n$-matrices) defines the covariant functor, i.e. the left $\fM$-module 
$$\underline{\cO}(M_n)\,:\,\fM \to {\rm Mod}_k\,,\qquad \langle m \rangle \mapsto \cO[M_n(k)]^{\otimes m}\ . $$
where $\cO[M_n(k)]$ denotes the coordinate ring of $ M_n $. On the other hand, any simplicial monoid $N$ 
yields a right $\fM$-module $k[N]$ with $\langle n \rangle \mapsto k[N]^{\otimes n}$, where $ k[N]$ is the
monoid algebra of $N$. Thus we may define
$$
\HR_\ast^{\rm Mon}(N, M_n(k))\,:=\, {\rm Tor}^{\fM}_\ast(k[N], \underline{\cO}(M_n))\ . 
$$
The next proposition is a special case of a general comparison result proved in \cite{BR22b}.
\begin{proposition}\la{hrmonalg}
Let $k$ be a field. For any simplicial monoid $N$, there is a natural isomorphism
\begin{equation} \la{monalgcomp} 
\HR_\ast^{\rm Mon}(N, M_n(k)) \,\cong\,\HR_\ast(k[N],k^n)\,, 
\end{equation}
where the right-hand side is the $n$-dimensional representation homology of the simplicial algebra $A = k[N]$ as defined in the Introduction.
\end{proposition}
\begin{proof}
Let $R \stackrel{\sim}{\to} N$ be a semi-free simplicial resolution. Since $k$ is a field, $k[R] \to k[N]$ is a weak equivalence of simplicial (right) $\fM$-modules by the K\"{u}nneth theorem. As in the proof of \cite[Theorem 4.2]{BRYI}, the right $\fM$-module $k[R]_m = k[R_m]$ of $m$-simplices is flat, since the module $k[S]$ corresponding to any {\it finitely generated} free monoid $S$ is projective, and $k[{R}_m]$ is a colimit of such modules. As in {\it loc. cit.}, it follows that 
\begin{equation} \la{ehradf1} \HR_\ast^{\rm Mon}(N, M_n(k)) \,\cong\, \pi_\ast[k[R] \otimes_{\fM} \underline{\cO}(M_n)]\ .
\end{equation}
Using the same argument as in \cite[Prop. 4.1]{BRYI}, it is easy to show that the tensor product in the right-hand side of \eqref{ehradf1} is naturally isomorphic to $\,(k[R])_n$, where $ (\,\mbox{--}\,)_n $ denotes the (degreewise simplicial extension of) representation functor \eqref{adj1}. Since the algebra $k[R] $ is a semi-free simplicial resolution of $k[N]$, we have
$\HR_\ast(k[N],k^n) \cong \pi_\ast[(k[R])_n]$ by definition \eqref{RHDefAlg}. This proves the desired proposition.
\end{proof}
Next, we compare the representation homology of (simplicial) monoids to that of (simplicial) groups.
\begin{proposition} \la{grpmonrh}
Let $k$ be a field of characteristic $0$. Then, for any simplicial group $\Gamma \in \sGr$, there is a natural isomorphism
\begin{equation} \la{grpmon}\HR_\ast(\Gamma,\GL_n(k)) \,\cong\, \HR_\ast^{\rm Mon}(\Gamma, M_n(k))\ .\end{equation}
\end{proposition}
\begin{proof}
We have
$$
\HR_\ast^{\rm Mon}(\Gamma, M_n(k)) := \Tor^{\fM}_\ast(k[\Gamma], \underline{\cO}(M_n)) =
\Tor^{\fM}_\ast(p^{\ast}\underline{k[\Gamma]}, \underline{\cO}(M_n))
\cong {\Tor}^{\ffgr}_\ast(\underline{k[\Gamma]},\,p_!\underline{\cO}(M_n))\,,
$$
where the last isomorphism is due to Lemma \ref{torpsh}. Proposition follows now from 
Lemma \ref{lpshrk}, which we prove below.
\end{proof}
\begin{lemma} \la{lpshrk}
Let $k$ be a field of characteristic $0$. There is an isomorphism of left $\ffgr$-modules in the 
derived category $ {\mathscr D}(\Mod_k(\ffgr))$:
\begin{equation} \la{lpshomn}\L p_!\underline{\cO}(M_n) \cong \underline{\cO}(\GL_n) \ .\end{equation}
\end{lemma}
\begin{proof}
Note that, for each $m \ge 0 $, there are natural isomorphisms in $ {\mathscr D}(k)$
$$
\L p_! \underline{\cO}(M_n)(\langle m \rangle)\, \cong\,  \underline{k \langle m \rangle} \otimes^{\L}_{\ffgr} \L p_! \underline{\cO}(M_n) \, \cong \, p^{\ast}\underline{k \langle m \rangle} \otimes^{\L}_{\fM} \underline{\cO}(M_n)\ .
$$
The first isomorphism above follows from the fact that  $k\langle m \rangle$ is projective as a right $\ffgr$-module; the second one is a consequence of Lemma \ref{torpsh}. Now, by Proposition \ref{hrmonalg}, we conclude
$$\H_\ast[p^{\ast}\underline{k \langle m \rangle} \otimes^{\L}_{\fM} \underline{\cO}(M_n)] \, =:\, \HR_\ast^{\rm Mon}(p^{\ast}\langle m \rangle, M_n) \,\cong\, \HR_\ast(k\langle x_1^{\pm 1},\ldots,x_m^{\pm 1} \rangle, k^n) $$
Since the algebra $\,k\langle x_1^{\pm 1},\ldots,x_m^{\pm 1} \rangle\,$ is quasi-free (formally smooth in the category $\Alg_k$), when $k$ is a field of characteristic zero (see, e.g., \cite[Prop. 5.3(2)]{CQ95}) and has a semi-free resolution with finitely many generators in each homological degree, we can use \cite[Theorem 21]{BFR} to compute
$$
\HR_i(k\langle x_1^{\pm 1},\ldots,x_m^{\pm 1} \rangle, k^n) \,=\, \begin{cases}
\, \cO[\GL_n(k)^m] & \mbox{if}\ i=0\\*[1ex]
                                                               \,   0 & \mbox{if}\ i>0 
                                                                 \end{cases}
$$
It follows that $\,\L p_! \underline{\cO}(M_n)(\langle m \rangle) \cong \cO[\GL_n(k)^m] \cong \underline{\cO}(\GL_n)(\langle m \rangle)$ for all $m \geqslant 0$. This implies \eqref{lpshomn}, completing the proof of the lemma.
\end{proof}
As a consequence of Proposition~\ref{hrmonalg} and Proposition~\ref{grpmonrh}, we get the
following result that identifies group representation homology with the representation homology of its group algebra. 
\begin{theorem} \la{grouptoalg}
Let $k$ be a field of characteristic $0$. For any $\Gamma \in \sGr$, there is a natural isomorphism
$$\HR_\ast(\Gamma,\GL_n(k)) \,\cong\, \HR_\ast(k[\Gamma],\,k^n)\ . $$
\end{theorem}
\begin{proof}
Indeed, $\,\HR_\ast(\Gamma,\GL_n(k)) \,\cong\, \HR_\ast^{\rm Mon}(\Gamma, M_n(k)) \,\cong\, \HR_\ast(k[\Gamma], k^n)\,$,
where the first isomorphism is by Proposition \ref{grpmonrh}, and the second by 
Proposition \ref{hrmonalg}.
\end{proof}
\begin{remark}
\la{Remgg}
Although not stated explicitly, the result of Theorem~\ref{grouptoalg} can be deduced from the proof of \cite[Theorem 3.1]{BRYI} for $ G = \GL_n(k) $ (see {\it loc. cit.},
Sect. 3.3). Our arguments above give a different, more direct proof of this result.
\end{remark}
\subsection{Symmetric bar construction}
\la{S2.3}
We recall (see \cite[6.1.4]{L})
that the {\it symmetric category} $\Delta S$ is defined to be the extension of $\Delta$ having the same objects as $\Delta$ (and $\Delta C$), with morphisms satisfying the two properties:
\vspace{1ex}
\begin{enumerate}
\item[(1)] For each $n \geqslant 0$, ${\rm Aut}_{\Delta S}([n]) \cong \Sigma_{n+1}^{\rm op}\,$, where $\Sigma_{n+1}$ is the $(n+1)$-th symmetric group.

\item[(2)] Any morphism $f:[n] \to [m]$ in $\Delta S$ can be factored uniquely as the composite $f=g \circ \sigma $ with $g \in {\rm Hom}_{\Delta}([n],[m])$ and $\sigma \in {\rm Aut}_{\Delta S}([n]) \cong \Sigma_{n+1}^{\rm op}$.
\end{enumerate}

\vspace*{1ex}

It is convenient use the following notation for the morphisms in $\Delta S$ (see \cite[Sect. 1.1]{Au1}). Any $\,f:\, 
[n] \xrightarrow{\sigma} [n] \xrightarrow{g} [m]\,$ in $ \Delta S $ can be written uniquely as the `tensor product' of $m+1$ noncommutative monomials $\,X_0,\,X_1,\,\ldots,X_m\, $ in $n+1$ formal variables $\{x_0,x_1,\ldots,x_n\}$:
\begin{equation} \la{s4.2e4} f= X_0 \otimes X_1 \otimes \ldots \otimes X_m \end{equation}
in such a way that each $X_i$ is the product $\,x_{i_1}x_{i_2} \ldots x_{i_r}\,$ of $r=|f^{-1}(i)|$ variables whose indices $ i_k $ occur in the fibre $f^{-1}(i)$ and 
that are ordered in the same way as numbers in $\{\sigma(0),\ldots,\sigma(n)\}$, i.e., $\sigma(i_1) < \sigma(i_2) < \ldots < \sigma(i_r)$. 

Using the above notation, the {\it symmetric bar construction} of an algebra $A \in \Alg_k$ is defined as the functor $\,B_{\rm sym}A: \Delta S \to \Mod_k\,$ that is given on objects by
\begin{equation}
\la{symbarob} 
B_{\rm sym}A([n]) = A^{\otimes (n+1)}\ ,\quad \forall\,n\ge 0\,, 
\end{equation}
and on morphisms by the following `substitution formula': 
\begin{equation} \la{symbarmor} 
B_{\rm sym}A(f)(a_0 \otimes \ldots \otimes a_n) \,:=\, (a_{i_1}\ldots a_{i_r}) \otimes \ldots \otimes (a_{k_1} \ldots a_{k_s})\ . \end{equation}
where $ f \in \Hom_{\Delta S}([n], [m])$ is represented by
$\,f = (x_{i_1}\ldots x_{i_r}) \otimes \cdots \otimes (x_{k_1} \ldots x_{k_s})\,$
({\it cf.} \cite[6.1.12]{L}). 

The symmetric bar construction extends naturally to differential graded and simplicial $k$-algebras: for $A \in \DGA_k$, we have $B_{\rm sym}A \in {\rm dgMod}_k(\Delta S) \cong {\rm Ch}({\rm Mod}_k(\Delta S))$ and for $ A \in \sAlg_k$, $\,B_{\rm sym}A \in {\rm sMod}_k(\Delta S) \cong {\rm Ch}({\rm Mod}_k(\Delta S))$, where in the last case the identification is given via the Dold-Kan normalization functor $N$. In this way, using
the same formula \eqref{HSA}, we can extend the definition of  symmetric homology
$\HS_*(A) $ to both DG and simplicial algebras.

For cocommutative Hopf algebras, we can give the following interpretation 
of the symmetric bar construction. By \cite[Lemma 4.2]{BR22}, there is a natural functor
\begin{equation} \la{s4.2e5} \Psi_{\rm sym}:\ \Delta S \to \ffgr^{\rm op}\end{equation}
defined on objects by  $\Psi_{\rm sym}([n]) = \langle n+1 \rangle $ and on morphisms by the following formula: if $ f \in \Hom_{\Delta S}([n], [m])$ is represented by $
f = (x_{i_1}\ldots x_{i_r}) \otimes \cdots \otimes (x_{k_1} \ldots x_{k_s})\ ,
$, then 
\begin{equation} \la{s4.2e6} \Psi_{\rm sym}(f): \langle m+1 \rangle \to \langle n+1 \rangle \,,\qquad X_0 \mapsto x_{i_1} \ldots x_{i_r}\,,\ \ldots\ ,\,X_m \mapsto x_{k_1} \ldots x_{k_s}\ ,
\end{equation}
where $\langle m+1 \rangle = \mathbb{F}\langle X_0,\ldots,X_m \rangle$ and $\langle n+1 \rangle = \mathbb{F}\langle x_0,\ldots,x_n\rangle$. Now, every (simplicial/DG) cocommutative Hopf algebra $\cA$ defines DG right $\ffgr$-module $\underline{\cA}$ by formula \eqref{rgmod}. As in the case of group algebras (see \cite[Lemma 3.3]{BR22}), it is straightforward to verify that
\begin{equation} 
\la{pullback} B_{\rm sym}\cA \,=\, \Psi_{\rm sym}^{\ast}(\underline{\cA}) \ ,
\end{equation}
where $\Psi_{\rm sym}^{\ast}$ denotes the restriction functor via \eqref{s4.2e5}. 

\section{Proofs of Main Results} 
\la{S3}
\subsection{Derived character maps}
\la{S3.1}
We now construct derived character maps relating symmetric homology to representation homology. As in the Introduction, we will work with simplicial algebras.

Let $R \stackrel{\sim}{\to} A$ be a semi-free simplicial resolution of $A \in  \sAlg_k$. Applying the unit of the adjunction \eqref{adj2} to $R$, we obtain a map of simplicial $k$-algebras $\,\pi_n: R \to M_n[(R)_n]\,$. Composing $\pi_n$ with the trace $\,M_n[(R)_n] \to (R)_n\,$ gives a map of simplicial $k$-modules $\,R \to (R)_n \,$. The latter
factors through the quotient $\,R/[R,R]\,$, inducing the simplicial $k$-linear map
\begin{equation}\la{TrR}
\Tr_n(R):\ R/[R,R] \to (R)_n \ .
\end{equation}
\begin{lemma} \la{hc0}
There is a natural isomorphism of simplicial $k$-modules 
$$
k \otimes_{\Delta C^{\rm op}} B^{\rm cyc}R\, \cong\, R/[R,R] \ .
$$
\end{lemma}
\begin{proof}
It suffices to check that $\,k \otimes_{\Delta C^{\rm op}} B^{\rm cyc}R_m \cong R_m/[R_m, R_m]\,$ for each $m$, but this follows immediately from formula \eqref{HCA} and the fact that $\HC_0(B) \cong B/[B,B]$ for any $B \in \Alg_k$. 
\end{proof}
By Lemma \ref{hc0}, we obtain the following composition of morphisms in $\mathscr{D}(k)$, where ${\rm C}$ stands for the associated chain complex:
\begin{equation} \la{derivedch1} k \otimes^{\L}_{\Delta C^{\rm op}} {\rm C}(B^{\rm cyc}R) \to  k \otimes_{\Delta C^{\rm op}} {\rm C}(B^{\rm cyc}R) \cong {\rm C}(R/[R,R]) \, \xrightarrow{\Tr_n(R)}\, {\rm C}[(R)_n] \,,
\end{equation}
where the last map is defined by \eqref{TrR}. Note that, if $k$ is a field, $B^{\rm cyc}R$ is weakly equivalent to $B^{\rm cyc}A$ by K\"{u}nneth's Theorem. In this case, \eqref{derivedch1} induces on homology a natural graded $k$-linear map
\begin{equation} 
\la{derivedch2}
\Tr_n(A)_\ast: 
\HC_\ast(A) \to \HR_\ast(A,k^n)\ , 
\end{equation}
which is the {\it derived character map} \eqref{Tr*}  referred to in the Introduction.
On the other hand, there is an inclusion functor $i:\Delta C^{\rm op} \hookrightarrow \Delta S$ defined in  \cite[6.1.11]{L}, for which it is easy to check 
\begin{equation} \la{cyctosym}B^{\rm cyc}A = i^{\ast} B_{\rm sym} A  \end{equation}
for any algebra $A \in \sAlg_k$. Hence, the functor $i$ induces a natural map 
\begin{equation} \la{hctohs} i^{\ast}:\ \HC_\ast(A) = {\rm Tor}^{\Delta C^{\rm op}}_{\ast}(k, B^{\rm cyc}A) \,\to\, {\rm Tor}^{\Delta S}_\ast(k, B_{\rm sym}A) = {\rm HS}_\ast(A)\ .\end{equation}
relating cyclic homology to symmetric homology. This is the map \eqref{HCS} in the Introduction.

Next, we recall that, for any algebra $ A \in \Alg_k$,
Ault \cite[Theorem 4.1]{Au2} constructed a graded commutative $k$-algebra structure on $\HS_\ast(A)$. The following lemma extends this structure to simplicial algebras, i.e., to any $A \in \sAlg_k$.
\begin{lemma} \la{algstr}
Ault's structure of a graded commutative algebra on ${\rm HS}_\ast(A)$ extends naturally to simplicial $k$-algebras $($defined over an arbitrary commutative ring $k)$.  
\end{lemma}
\begin{proof}
Recall (see \cite[Sect. 3]{Au2}) that Ault constructed a functor $C(\mbox{--})_\ast:\Alg_k \to \mathrm{sE}_\infty\Alg_k$ such that  $\pi_\ast[CA_\ast] \cong {\rm HS}_\ast(A)$ for any
$A \in \Alg_k$. Here, $\mathrm{sE}_\infty\Alg_k$ denotes the category of $E_{\infty}$-algebra objects in the category of simplicial $k$-modules (which we shall refer to as {\it simplicial $E_{\infty}$-algebras}). Here, the term `$E_{\infty}$-algebra' refers to an algebra over the simplicial $E_{\infty}$-operad $\mathcal{D}_{\bf Mod}$ of \cite[Remark 2.8]{Au2}. Now let $A \in \sAlg_k$ be an arbitrary simplicial $k$-algebra. Then, $CA$ is a bisimplicial $E_\infty$-algebra whose diagonal $k$-module is an $E_{\infty}$-algebra. Let $\Delta S_+$ be the category obtained by adjoining an initial object $[-1]$ to $\Delta S$. Now recall that, by definition,  
$$CA = k[N(\mbox{--}\downarrow \Delta S_+)] \otimes_{\Delta S_+} B_{\rm sym,+}A \,,
$$
where $N(\mbox{--})$ stands for the simplicial nerve functor and $B_{\rm sym,+}A$ is the extension of $B_{\rm sym}A$ to a $\Delta S_+$-module as in \cite[Def. 30]{Au1}. Since the simplicial right $\Delta S_+$-module $k[N(\mbox{--}\downarrow \Delta S_+)]$ is a levelwise projective resolution of $k$, the homology of (the Dold-Kan normalization of) the diagonal of $CA$ is precisely
$${\rm Tor}^{\Delta S_+}_\ast(k, B_{\rm sym,+}A)\, \cong\, {\rm Tor}^{\Delta S}_\ast(k, B_{\rm sym}A) = : \HS_\ast(A)\ . $$
The above isomorphism  follows from \cite[Theorem 31]{Au1}. This proves the desired lemma.
\end{proof}

Now, we come to the main result of this section.
\begin{proposition} \la{factorTr}
Assume that $k$ is a field. The one-dimensional derived character map \eqref{derivedch2} factors through \eqref{hctohs}:
\begin{equation*}
 \la{HSR2} 
 \begin{diagram}[small]
 \HC_*(A)& & \rTo^{\ \Tr(A)_\ast\ } & & \HR_*(A, k)\\
& \rdTo_{\iota^*} &                         & \ruDotsto &\\
&                  & \HS_*(A)        & &
 \end{diagram}
 \end{equation*}
inducing a {\rm homomorphism} of graded commutative $k$-algebras with respect to the Ault structure on $\HS_*(A)$:
\begin{equation}
\la{SRAA}
\SR(A)_*:\ \HS_\ast(A)\, \to\, \HR_*(A, k) \ 
\end{equation}
\end{proposition}
\begin{proof}
Given an algebra $A$, we can construct the map \eqref{SRAA} as follows. Fix a semi-free simplicial resolution $R \stackrel{\sim}{\to} A$ in $\sAlg_k $ and consider the natural map
\begin{equation}\la{dttot} k \otimes^{\L}_{\Delta S} {\rm C}(B_{\rm sym} R) \to  k \otimes_{\Delta S} {\rm C}(B_{\rm sym} R)  \end{equation}
where ${\rm C}$ stands for the associated chain complex. By the K\"{u}nneth Theorem, $B_{\rm sym}R \xrightarrow{\sim} B_{\rm sym}A$ is a weak equivalence. Hence $\H_\ast[k \otimes^{\L}_{\Delta S} {\rm C}(B_{\rm sym} R)] \cong \HS_\ast(A)$. By \cite[Theorem 86]{Au1}, there is a natural isomorphism of simplicial $k$-vector spaces $k \otimes_{\Delta S} B_{\rm sym} R \cong R_{\rm ab} =: (R)_1$. Hence $\H_\ast[k \otimes_{\Delta S} {\rm C}(B_{\rm sym} R)] \cong \HR_\ast(A,1)$.Thus, \eqref{dttot} induces on homology a graded $k$-linear map $\SR_\ast(A):\HS_\ast(A) \to \HR_\ast(A,1)$. 
 
 Next, note that the map $\Tr_1(R):R/[R,R] \to (R)_1$ inducing $\Tr(A)_\ast$ is just the canonical quotient to the abelianization, and the map $i^{\ast}: k \otimes_{\Delta C^{\rm op}} B^{\rm cyc}R \to k \otimes_{\Delta S} B_{\rm sym} R$ coincides with $ \Tr_1(R) $. Hence the following diagram commutes in 
 $ {\mathscr D}(k) $: 
 $$  
 \begin{diagram}[small]
   k \otimes^{\L}_{\Delta C^{\rm op}} {\rm C}(B^{\rm cyc}R)  & \rTo^{i^{\ast}} & k \otimes^{\L}_{\Delta S} {\rm C}(B_{\rm sym} R)\\
    \dTo & & \dTo\\
   k \otimes_{\Delta C^{\rm op}} {\rm C}(B^{\rm cyc}R) & \rTo^{\Tr_1(R)} &  k \otimes_{\Delta S} {\rm C}(B_{\rm sym} R)\\
 \end{diagram}
 $$
 which gives the natural isomorphism $ \SR_*(A) \circ i^* \cong \Tr(A)_* $ on homology.
 
 It remains to prove that the map \eqref{SRAA} constructed above is an algebra homomorphism. To this end,
consider the bisimplicial $E_{\infty}$-algebra $C{R}_\ast$, where $C(\mbox{--})_\ast:\Alg_k \to \mathrm{sE}_\infty\Alg_k$ is Ault's functor (see the proof of Lemma \ref{algstr}). The total (i.e. diagonal) simplicial vector space of $CR_\ast$ is a (simplicial) $E_{\infty}$-algebra. By (the proof of) \cite[Lemma 3.1]{Au2}, 
$$(C{R})_i\ \cong \bigoplus_{[n] \stackrel{\rm id}{\to} [n] \stackrel{f_1}{\to} [n_1] \stackrel{f_2}{\to} \ldots \stackrel{f_i}{\to} [n_i]}  {R}^{\otimes i}\,, $$
where the direct sum runs over the chains of morphisms $[n] \stackrel{\rm id}{\to} [n] \stackrel{f_1}{\to} [n_1] \stackrel{f_2}{\to} \ldots \stackrel{f_i}{\to} [n_i]$ of length $i+1$ in the category $\Delta S_+$ obtained by adjoining an initial object $[-1]$ to $\Delta S$.  Since
$$ (C{R})_i\ \cong \bigoplus_{[n] \stackrel{\rm id}{\to} [n] \stackrel{f_1}{\to} [n_1] \stackrel{f_2}{\to} \ldots \stackrel{f_i}{\to} [n_i]}  {R}^{\otimes i}\,,$$
and
$$ (C{A})_i\ \cong \bigoplus_{[n] \stackrel{\rm id}{\to} [n] \stackrel{f_1}{\to} [n_1] \stackrel{f_2}{\to} \ldots \stackrel{f_i}{\to} [n_i]}  {A}^{\otimes i}\,,$$
it follows from the K\"{u}nneth theorem that $(CR)_i \to (CA)_i$ is a weak equivalence of simplicial $E_{\infty}$-algebras for all $i$.  Hence, ${\tt diag}(C{R})_\ast \xrightarrow{\sim} {\tt diag}(CA)_\ast$ is a quasi-isomorphism of $E_{\infty}$-algebras. On the other hand, the natural map $C(R_i) \to \HS_0(R_i) \cong (R_i)_{\rm ab}$ is an $E_{\infty}$-algebra homomorphism by \cite[Proposition 4.4]{Au2}, where $R_i$ is the algebra of $i$-simplices on $R$. 
Hence the map ${\tt diag}(CR) \to R_{\rm ab}$ is a map of simplicial $E_{\infty}$-algebras.
Since the map $C(R_i) \to \HS_0(R_i) \cong (R_i)_{\rm ab}$ represents the map $k \otimes^{\L}_{\Delta S} B_{\rm sym}R_i \to k \otimes_{\Delta S} B_{\rm sym}R_i$, the map induced on homology (of the associated chain complexes) by the map ${\tt diag}(CR) \to R_{\rm ab}$ is the same as that induced on homologies by \eqref{dttot}, which is $\SR_\ast(A)$. Since the map ${\tt diag}(CR) \to R_{\rm ab}$ is a homomorphism of simplicial $E_{\infty}$-algebras, $\SR_\ast(A)$ is a morphism of graded commutative algebras. This finishes the proof of the proposition.
\end{proof}
\begin{remark} \la{generalization}
The result of Proposition \ref{factorTr} holds true in greater generality: for any algebra $A$ that is flat over a commutative ring $k$. This last condition is automatic when $k$ is a field.
\end{remark}
\begin{remark}
\la{Nogo}
In general, for $n>1$, the derived character map $ \Tr_n(A)_\ast: \HC_*(A) \to \HR_*(A, k^n) $ does not seem to factor through $\HS_*(A)$. 
\end{remark}
%

\subsection{Symmetric homology vs representation homology} 
\la{S3.2}
{\it From now on, we will assume that $k$ is a field of characteristic $0$.}
Theorem \ref{T1} stated in the Introduction  is a special case of the following more general result.
\begin{theorem} \la{T1a}
For any simplicial algebra $A \in \sAlg_k$, the map $\SR_\ast(A)$ of Proposition~\ref{factorTr} is an isomorphism of graded commutative $k$-algebras.
\end{theorem}
The proof of Theorem~\ref{T1a} is based on the following key proposition.

\begin{proposition} \la{pC}
Let $V$ be a $k$-vector space, and let $TV = T_k(V)$ denote the tensor algebra of $V$. Then
$$
{\rm Tor}^{\Delta S}_i(k, B_{\rm sym}TV)\, \cong\, \begin{cases}
                                                    \Sym_k(V) & i =0 \\
                                                     0 & i\neq 0
                                                  \end{cases}
                                                  $$
\end{proposition}
We will give two different proofs of Proposition \ref{pC}, both relying on topological results. Our first proof is parallel to that of \cite[Proposition 2.1]{BRYII}: 
it is based on \cite[Theorem 3.2]{BRYII}, its Corollary \cite[Corollary 3.2]{BRYII} as well as 
\cite[Proposition 5.3]{BRYI} and \cite[Corollary 5.1]{BR22}, the last two  involving topological arguments.
\begin{proof}[First proof of Proposition \ref{pC}]
By \cite[Theorem 86]{Au1}, we know that 
$\,\HS_0(A) \cong A_{\rm ab}\,$ for any $A \in \Alg_k$. Hence  $\,k \otimes_{\Delta S} B_{\rm sym}TV \cong \Sym_k(V)\,$. We need only to prove the vanishing of higher homology:
$${\rm Tor}^{\Delta S}_i(k, B_{\rm sym}TV) = 0 \quad \text{for all}\quad i>0\,
$$ 
To this end we assign $V$ homological degree $2$, and let $\cL V$ denote the free graded Lie algebra generated by $V$. Then $TV \cong U(\cL V)$, where $U(\mbox{--})$ stands for universal enveloping algebra. This gives $TV$ the structure of a (graded) cocommutative Hopf algebra. Let $\underline{TV}$ denote the corresponding graded right $\ffgr$-module, whose component in homological degree $2q$ is $\underline{TV}_q$, the component of $\underline{TV}$ of weight $q$. More explicitly, $\underline{TV}_q(\langle m \rangle)$ is the span (over $k$) of elements of the form $w_{d_1}\otimes \ldots \otimes w_{d_m} \in (TV)^{\otimes m}$ such that $d_1 + \ldots +d_m =q$ and $w_{d_i} \in V^{\otimes d_i} \subset TV$. By \eqref{pullback}, $B_{\rm sym}TV$ acquires the structure of a graded $\Delta S$ module whose component in homological degree $2q$ is 
$$ (B_{\rm sym}TV)_q\,:=\, \Psi_{\rm sym}^{\ast}\underline{TV}_q\ .$$ 
Thus, 
$$ \H_n[k \otimes^{\L}_{\Delta S} B_{\rm sym}TV] \,\cong\, \Moplus_{2q+i=n} {\rm Tor}^{\Delta S}_i(k, (B_{\rm sym}TV)_q)\ .$$
The desired proposition will follow once we show that 
\begin{equation} \la{symhtv} \H_\ast[k \otimes^{\L}_{\Delta S} B_{\rm sym}TV] \,\cong\, \Lambda_k(V)\ . \end{equation}
By \cite[Theorem A.1]{BRYII}, there are Quillen equivalences refining the Dold-Kan equivalence $N:\Com_k \rightleftarrows \sVect_k:N^{-1}$
$$N^{\ast}\,:\,\DGL_k^{+} \rightleftarrows \sLie_k\,:\,N\,,\qquad  N^{\ast}\,:\,\DGA_k^{+} \rightleftarrows \sAlg_k\,:\,N  $$
where $\mathtt{s}\cC$ denotes the category of simplicial objects in a category $\cC$. Equip $N^{\ast}(TV) \cong T(N^{-1}V)$ (see formula A.4 in the Appendix of {\it loc. cit.}) with simplicial cocommutative Hopf algebra structure given by its identification with the universal enveloping algebra $U(\cL(N^{-1} V))$ of the free Lie algebra generated by $N^{-1}V$. We denote the corresponding simplicial right $\ffgr$-module by $\underline{N^{\ast}TV}$. Let $B_{\rm sym}N^{\ast}(TV) = \Psi_{\rm sym}^{\ast}(\underline{N^{\ast}TV})$. Since $V$ has homological degree $2$, $N^{\ast}\cL V \cong \cL(N^{-1}V)$ (see formula A.4 in the appendix of \cite{BRYII}) is a semi-free simplicial model of the space $X$ given by the wedge of ${\rm dim}_k V$ copies of the three sphere $\mathbb{S}^3$. By \cite[Theorem 3.2]{BRYII} the right $\ffgr$-modules $k[\lgr({X})]$ and $\underline{N^{\ast}TV}$ are (naturally) weakly equivalent, where $\lgr({X})$ is the Kan loop group of $X$. Hence $B_{\rm sym}N^{\ast}TV$ and $B_{\rm sym}k[\lgr(X)]$ are (naturally) weakly equivalent as simplicial $\Delta S$-modules. As in {\it loc. cit.}, Corollary 3.2, it follows that there is a natural isomorphism on homologies
$$\H_\ast[k \otimes^{\L}_{\Delta S} N(B_{\rm sym}N^{\ast}TV)] \,\cong\, \H_\ast[k \otimes^{\L}_{\Delta S} N(k[\lgr(X)])]\ . $$

The right hand side is isomorphic to the symmetric homology ${\rm HS}_\ast(X)$ of a(ny) simplicial group model of $X$ by definition (see \cite[Section 4.2]{BR22}). Hence there are isomorphisms
\begin{equation} \la{hsspace_x} \H_\ast[k \otimes^{\L}_{\Delta S} N(B_{\rm sym}N^{\ast}TV)] \,\cong\, {\rm HS}_\ast(X) \,\cong\, \HR_\ast(X, \mathbb{G}_m(k)) \,\cong\, \Lambda_k(V)\ . \end{equation}
The second isomorphism above is by \cite[Corollary 5.1]{BR22} (since $X$ is indeed simply connected), while the last isomorphism above follows from our earlier computation of the representation homology of wedges of spheres \cite[Proposition 5.3]{BRYI}. Finally, as observed in (the proof of) \cite[Proposition 2.1]{BRYII}, there is a quasi-isomorphism of right $\ffgr$-modules $\underline{\varepsilon}:\underline{TV} \to N(\underline{N^{\ast}TV})$ induced by the unit of the adjunction between the functors $N$ and $N^{\ast}$. It follows that $\Psi^{\ast}_{\sym}(\underline{\varepsilon}): B_{\rm sym} TV \to N(B_{\rm sym}N^{\ast} TV)$ is a quasi-isomorphism of $\Delta S$-modules. By \eqref{hsspace_x}, $\H_\ast[k \otimes^{\L}_{\Delta S} B_{\rm sym}TV] \,\cong\, \bLambda_k V$. This verifies \eqref{symhtv}, completing the proof of the desired proposition.
\end{proof}
Our second proof of Proposition \ref{pC} is shorter than the first, but it relies heavily on results of Fiedorowicz and Ault (in particular, on rather technical computations of Lemma~33 and Lemma~36 of \cite{Au1}).
\begin{proof}[Second proof of Proposition \ref{pC}]
Choose a basis in $V$, say a set $X \subset V $, so that $TV \cong T(k[X]) \cong k[J(X_+)]$, where $J(X_{+})$ denotes the free monoid generated by the set $ X_{+} = X \amalg \{\ast\}$, with `$ \ast $' identified with the identity element. We regard $J(X_{+})$
as the classical James construction of the discrete space $X$. Then, by \cite[Lemma 33]{Au1}, we have 
$$
{\rm HS}_\ast(TV)\, \cong\, {\rm HS}_\ast(k[J(X_+)])\, \cong \, \H_\ast[\mathcal{N}X; k]\ ,
$$ 
where 
$$\mathcal{N}X \,:=\, \coprod_{n \geq -1} N([n] \downarrow \Delta S_+) \times_{\Sigma_{n+1}^{\rm op}} X^{n+1}\ . $$
Here, $\Delta S_+$ is the category obtained by adding to $\Delta S$ the initial object $[-1]$, and $N(\mbox{--})$ in the formula above stands for the nerve. By (the proof of) \cite[Lemma 36]{Au1}, there is a weak equivalence $\mathcal{N}X \simeq \mathcal{D}X$, where $\mathcal{D}X$ is the free $\mathcal{D}$-algebra generated by $X$ for a specific simplicial $E_{\infty}$-operad $\mathcal{D}$. Hence 
$$ {\rm HS}_\ast(TV) \,\cong\, \H_\ast[\mathcal{N}X; k] \,\cong\,\H_\ast[\mathcal{D}X; k]  \ .$$
Now, recall that $\mathcal{D}X= \coprod_{n \geqslant 0} E_{\ast}\Sigma_{n}^{\rm op} \times_{\Sigma_n} X^n$. Hence, 
$$\H_\ast[\mathcal{D}X; k] \,\cong\, \Moplus_{n=0}^{\infty} k[E_{\ast}\Sigma_{n}^{\rm {op}}] \otimes_{k[\Sigma_n]} k[X]^{\otimes n} \,\cong\, \Moplus_{n=0}^{\infty} k[E_{\ast}\Sigma_{n}^{\rm {op}}] \otimes_{k[\Sigma_n]} V^{\otimes n}\ . $$
The last isomorphism because $X$ is a discrete basis of $V$.  Since $E_{\ast}\Sigma_n^{\rm op}$ is contractible for all $n$, $k[E_\ast \Sigma_n^{\rm op}]$ is weakly equivalent to $k$. Further, $k[E_\ast \Sigma_n^{\rm op}]$ is a levelwise projective simplicial right $\Sigma_n$-module. Hence,
$$  {\rm HS}_\ast(TV) \,\cong\, \H_\ast[\mathcal{D}X; k] \,\cong\, \Moplus_{n=0}^{\infty} \H_\ast[k \otimes^{\L}_{k[\Sigma_n]} V^{\otimes n}] \,\cong\, \Moplus_{n=0}^{\infty} \H_\ast(\Sigma_n, V^{\otimes n})\,\cong\,\Sym_k(V)\ .$$
where the last isomorphism holds because  $k$ is a field of characteristic $0$. 
\end{proof}
Now, we can proceed with

\begin{proof}[Proof of Theorem \ref{T1a}]
Let $R \stackrel{\sim}{\to} A$ be a semi-free simplicial resolution. Then, the (left) $\Delta S$-module of $n$-simplices in $B_{\rm sym}{R}$ is of the form $B_{\rm sym}TV$ for some vector space $V$. It follows from Proposition \ref{pC} that the following 
natural chain map is a quasi-isomorphism:
\begin{equation} \la{srr} k \otimes^{\L}_{\Delta S} {\rm C}(B_{\rm sym}R) \to k \otimes_{\Delta S} {\rm C}(B_{\rm sym}R) \cong {\rm C}(R_{\rm ab})\,, \end{equation}
where last isomorphism is a consequence of \cite[Theorem 86]{Au1}. Since \eqref{srr} is precisely the map that induces $\SR_\ast(A)$ on homology, the $\SR_\ast(A)$ is an isomorphism of graded $k$-vector spaces.
By Proposition \ref{factorTr}, we then conclude that this isomorphism of vector spaces
is actually an isomorphism of $k$-algebras.
\end{proof}

As a first application of Theorem~\ref{T1a}, we state the following result that strengthens \cite[Corollary 5.1]{BR22} (see Remark~\ref{Rem2} in the Introduction). To avoid confusion
we recall that, unlike representation homology of group algebras, the representation homology of groups is defined in terms of Tor-functors (see Definition~\ref{RHDef}). 
However, when $\cH = \cO[\GL_n(k)] $ and $k$ is a field of characteristic $0$, Theorem~\ref{grouptoalg} proved in Section~\ref{S2.2} shows that the two definitions agree. In combination with Theorem~\ref{T1a}, this gives
\begin{corollary}
\la{GrHS}
For any simplicial group $ \Gamma \in \sGr $, there is a natural algebra isomorphism
$$
\HS_*(k[\Gamma])\,\cong\, \HR_*(\Gamma, \bG_m(k))\,,
$$
where $ \bG_m(k) = \GL_1(k) $ is the multiplicative group of $k$.
\end{corollary}

Next, we apply Theorem~\ref{T1a} to verify Conjecture \ref{Con2} stated in the Introduction. For a detailed explanation of the notation used in this conjecture, we refer to \cite[Theorem~8]{AF07} and \cite[Remark 2.8]{Au2}. Here, we only recall that $D$ denotes the monad associated to an $E_\infty$-operad in the category of simplicial $k$-modules that, for $ V \in \sMod_k $, can be explicitly defined by
\begin{equation}
\la{DV}
DV = \Moplus_{n=0}^{\infty} k[E_{\ast}\Sigma_{n}^{\rm {op}}] \otimes_{k[\Sigma_n]} V^{\otimes n}
\end{equation} 
Given a $k$-algebra $A$, $\, B(D,T,A)\,$ stands then for the two-sided topological bar construction, i.e. the geometric realization of a simplicial bar constriction
$ B_\ast(D,T,A)$, associated to the monad $D$ and the tensor algebra functor $T: \Mod_k \to \Mod_k $.

\begin{corollary} \la{bdta} 
For $A \in \Alg_k$, there is an isomorphism of graded commutative $k$-algebras 
$$
{\rm HS}_\ast(A) \cong \H_\ast[B(D,T,A)]\ .
$$ 
\end{corollary}
\begin{proof}
When $k$ is a field of characteristic $0$, we have shown (see the second proof of Proposition \ref{pC}) that there is a natural homology isomorphism $\,DV \to \Sym_k(V)\,$ for any $k$-vector space $V$. This induces a homology isomorphism of the bar constructions: $\,B(D,T,A) \to B(\Sym, T, A)\,$, where $\Sym$ is the monad associated with the commutative operad. Hence
$$
\H_\ast[B(D,T,A)] \,\cong\,  \H_\ast[B(\Sym,T,A)]
                \,    \cong \, \pi_\ast[B_{\ast}(\Sym, T, A)]
                  \,  \cong\,  \pi_\ast[B_\ast(T,T,A)_{\rm ab}]
                    \,\cong \, \HR_\ast(A,k)\,,
$$
where the last isomorphism follows from the well-known fact that $B_\ast(T,T,A)$ is a semi-free simplicial resolution of $A$. Theorem \ref{T1a} now implies the desired result.
\end{proof}

For further applications, it will be convenient for us to restate the result of Theorem~\ref{T1a} for differential graded algebras. To this end we will use monoidal versions
of the classical Dold-Kan correspondence in the form proved in \cite[Theorem A.1]{BRYII}. Specifically, if $k$ is a field of characterisitc zero, there are natural Quillen equivalences: 
\begin{equation} \la{refdk}N^{\ast}\,:\,\DGA_k^{+}\rightleftarrows \sAlg_k\,:\,N\,,\qquad N^{\ast} \,:\,\cDGA_k^{+} \rightleftarrows \scAlg_k\,:\,N\,, \end{equation}
where $ \DGA_k^+$ (respectively, $\cDGA_k^{+}$) denote the model categories of non-negatively graded (chain) DG $k$-algebra (respectively, commutative DG algebras) over $k$. In the above adjunctions, the right adjoint functors are given by the Dold-Kan 
normalization functor $N$ restricted to the corresponding subcategories of algebras in $\sMod_k$. Since $N$ reflects weak equivalences and the unit ${\rm Id} \to N N^{\ast}$ is a natural weak equivalence, the left adjoint functors $N^{\ast}$ preserve weak equivalences. 

Now, recall that, in \cite[Section 2]{BKR}, we defined the representation homology $\HR_\ast(A,k^n)$ for any $ A \in \DGA_k^+$. Apart from the model categories $\sAlg_k$ and $\scAlg_k$ being replaced by $\DGA_k^{+}$ and $\cDGA_k^+$ respectively, the definition of \cite{BKR} is identical to the definition of representation homology of simplicial algebras that we presented in the Introduction. To deduce the following proposition from Theorem \ref{T1a} we need only to check that the equivalence $ N^{\ast}: \DGA_k^{+} \to \sAlg_k$ preserves both symmetric and representation homology.
\begin{proposition}\la{tA}
For any $A \in \DGA_k^{+}$, there is an isomorphism of graded $k$-vector spaces $\HS_\ast(A) \cong \HS_\ast(N^{\ast}A)$. As a consequence, there is an isomorphism of graded commutative $k$-algebras
$${\rm HS}_\ast(A) \cong \HR_\ast(A,k)\,,$$
where $\HS_\ast(A)$ inherits its algebra structure from $\HS_\ast(N^{\ast}A)$.
\end{proposition} 
\begin{proof}
There are natural quasi-isomorphisms of complexes of $\Delta S$-modules
$$B_{\rm sym}A \stackrel{\sim}{\to} B_{\rm sym}(NN^{\ast}A) \stackrel{\sim}{\to} N(B_{\rm sym}N^{\ast} A)\,, $$
where the the first arrow --- induced by the natural weak equivalence $A \to NN^{\ast}A$ of DG algebras --- is a weak equivalence by the K\"{u}nneth Theorem, and the second arrow
is a Eilenberg-Zilber map, which is also known to be a quasi-isomorphism by \cite[Sect. A.2]{BRYII}.  Thus, on homology, we get natural isomorphisms of graded $k$-vector spaces
$$\HS_\ast(A):= \H_\ast[k \otimes^{\L}_{\Delta S} B_{\rm sym}A] \,\cong\, \H_\ast[k \otimes^{\L}_{\Delta S} N(B_{\rm sym}N^{\ast} A)] =: \HS_\ast(N^{\ast}A)\ . $$
proving the first statement in the desired proposition. 

Next, since the abelianization functor $(\,\mbox{--}\,)_{\rm ab}$ is left adjoint, it commutes with $N^{\ast}$: i.e., there is a commutative diagram of functors 
\begin{equation} \la{abnst}
\begin{diagram}[small]
\DGA_k^{+} & \rTo^{N^{\ast}} & \sAlg_k\\
 \dTo^{(\mbox{--})_{\rm ab}} & &  \dTo_{(\mbox{--})_{\rm ab}}\\
 \cDGA_k^{+} & \rTo^{N^{\ast}} & \scAlg_k
 \end{diagram}\ .
\end{equation}
Let $R \stackrel{\sim}{\to} A$ be a semi-free resolution of $A$ in $\DGA_k^{+}$. There are isomorphisms (that are natural in $R$)
\begin{equation} \la{nnstrh} \HR_{\ast}(A, k) \,\cong\, \H_\ast(R_{\rm ab}) \,\cong\, \pi_{\ast}[N^{\ast}(R_{\rm ab})] \,\cong\, \pi_\ast[(N^{\ast} R)_{\rm ab}]\,\cong\,\HR_\ast(N^{\ast}A, k)\ . \end{equation}
The first isomorphism above is by definition (see \cite[Sect. 2]{BKR}), the second isomorphism above is by \eqref{abnst}, and the last isomorphism above is by definition (see Definition \ref{RHDefAlg}) and the fact that $N^{\ast}R \to N^{\ast}A$ is a semi-free resolution of $A$ (by \cite[Prop. A.2]{BRYII} and since $N^{\ast}$ preserves weak-equivalences).  This shows that there is an isomorphism of graded commutative $k$-algebras
\begin{equation} \la{hrlnast} \HR_\ast(A,k) \cong \HR_\ast(N^{\ast}A, k)\ . \end{equation}
The second statement of the desired proposition now follows immediately from Theorem \ref{T1a}.
\end{proof}
To illustrate the usefulness of Proposition \ref{tA} we give one example.
\begin{example}
Let $ A =k[x]/(x^2)$ be the algebra of dual numbers defined over a field $k$
of characteristic $0$. It is known (see, e.g., \cite[Sect. 6.3.1]{BFR}) that
$A$ has the semi-free DG algebra resolution $\,R:= k\langle x, t_1,t_2,t_3, \ldots \rangle \,$, where $\deg(t_i) = i $ and the differential is given by $\,dt_1 = x^2 \,$ and
$$ 
dt_i = xt_{i-1}-t_1t_{i-2}+\ldots +(-1)^{i-1}t_{i-1}x \ ,\quad i\ge 2\,.  
$$
By Proposition \ref{tA}, ${\rm HS}_\ast(A)$ is the homology of the commutative DG algebra 
$k[x,t_1,t_2,t_3,\ldots]$ with differential given by the same formulas as above. 
Using \texttt{Macaulay2} computer algebra software, one can find the groups $ {\rm HS}_i(A) $ explicitly to sufficiently high homological degrees ({\it cf.} \cite{BFR}):
\begin{eqnarray*}
{\rm HS}_0(A) & \cong & A\\
{\rm HS}_1(A) & \cong & 0\\
{\rm HS}_2(A) & \cong & A \cdot t_2\\
{\rm HS}_3(A) & \cong & A\cdot (xt_3- 2t_1t_2)\\
{\rm HS}_4(A) & \cong & A \cdot t_2^2 \oplus A \cdot t_4\\
{\rm HS}_5(A) & \cong & A \cdot (-2t_1t_2^2+xt_2t_3) \oplus A \cdot (-t_2t_3-4t_1t_4+2xt_5)\\
{\rm HS}_6(A) & \cong & A \cdot t_2t_4 \oplus A \cdot t_6\\
{\rm HS}_7(A) & \cong & A \cdot (-t_2^2t_3-4t_1t_2t_4+2xt_2t_5) \oplus A \cdot (-t_3t_4-2t_1t_6+xt_7)\\
{\rm HS}_8(A) & \cong & A \cdot t_2t_6 \oplus A \cdot t_4^2 \oplus A \cdot t_8
\end{eqnarray*}
%
\end{example}

\subsection{Symmetric homology of universal enveloping algebras}
\la{S3.3}
As a main application of Proposition \ref{tA}, we now prove the following result that includes (as special cases) Theorem~\ref{TPoly} and Theorem \ref{T2} stated in the Introduction.
\begin{theorem} 
\la{hsua}
Let $\mfa \in \DGL^+_k$ be a chain DG Lie algebras over $k$, and let $ U\mfa \in \DGA^+_k $ denote its universal enveloping algebra.
There is an isomorphism of graded commutative $k$-algebras 
$$
{\rm HS}_\ast(U\mfa) \,\cong\, \Lambda_k[\rH_{\ast+1}(\mfa;k)]\,, 
$$
where $\rH_\ast(\mfa;k)$ is the $($reduced$)$ Chevalley-Eilenberg homology of $\mfa$ with trivial coefficients.
\end{theorem}
\begin{proof}
Recall (see, e.g., \cite[Sect. 6.2]{BFPRW}) that $U\mfa$ is Koszul dual to the cocommutative DG coalgebra $C:=\cC_\ast(\mfa;k)$, whose underlying chain complex is the Chevalley-Eilenberg complex of $\mfa$. This means that
the cobar construction $\cb(C) \stackrel{\sim}{\to} U\mfa$ yields a cofibrant DG algebra resolution of $U\mfa$ (see, e.g, \cite[Chap. 2]{LVO}). Recall that $\cb(C)$ is defined to be the tensor algebra $T_k(\bar{C}[-1])$ with differential being the sum of two degree $(-1)$ derivations: $\,d = d_1+d_2$. On generators, $d_1$ is defined to be the differential of $\bar{C}[-1]$, while $d_2$ is given on  $\bar{C}[-1]$  by the following composite map
$$\begin{diagram}[small] \bar{C}[-1] ] \cong k[-1] \otimes \bar{C} & \rTo^{\Delta_s \otimes \Delta} & k[-1] \otimes k[-1] \otimes \bar{C}^{\otimes 2} & \rTo^{\cong} & (k[-1] \otimes \bar{C})^{\otimes 2} \cong (\bar{C}[-1])^{\otimes 2} \subset T_k(\bar{C}[-1]) \end{diagram}\,,
$$
where $\Delta$ is the coproduct on $\bar{C}$ and $\,\Delta_s:k[-1] \to k[-1] \otimes k[-1]\,$ is the linear map taking $1_{k[-1]}$ to $-1_{k[-1]} \otimes 1_{k[-1]}$. 
Since $\bar{C}$ is cocommutative, $d_2$, in fact, maps $\bar{C}[-1]$ to $\wedge^2(\bar{C}[-1]) \subset \cb(C)$. Hence, 
\begin{equation} \la{cbab} 
\cb(C)_{\rm ab} \,\cong\, \Lambda_k(\bar{C}[-1])\,,\end{equation}
with the differential in the right hand side induced by $d_1=d_{\bar{C}[-1]}$. 

Now, we have the following chain of natural isomorphisms
\begin{equation}
\la{isoLie}
{\rm HS}_\ast(U\mfa) \,\cong\, \HR_\ast(U\mfa,k) \,\cong\, \H_\ast[\cb(C)_{\rm ab}] \,\cong\, \H_\ast[\Lambda_k(\bar{C}[-1])] \,\cong\, 
\Lambda_k[\rH_{\ast+1}(\mfa;k)]\,, 
\end{equation}
The first isomorphism in \eqref{isoLie} follows from Proposition \ref{tA}, the second is by  definition of representation homology for (non-negatively graded) DG algebras (see \cite[Sect. 2]{BKR}), the third is induced by \eqref{cbab}, and the last one is a consequence of the well-known fact that the symmteric algebra functor $\Lambda_k$ commutes with homology when $k$ is a field of characteristic $0$ (see, e.g. \cite[Appendix B, Prop. 2.1]{Q}).
\end{proof}
Note that if $ V $ is an $n$-dimensional vector space over $k$, we can think of it
as the Lie algebra $ \mfa = V $ with trivial bracket. Then $U\mfa = S_k(V) $ and
$$ 
\rH_i(\mfa;k) \,\cong\, \begin{cases}
                         \wedge^i V & 1 \leqslant i \leqslant n\\
                           0           & \text{ otherwise }
                       \end{cases}
$$                      
Hence, by Theorem \ref{hsua}, we have
\begin{equation}
\la{HSpoly}
{\rm HS}_\ast[S_k(V)] \,\cong\, \Lambda_k\big[ \Moplus_{i=1}^n \wedge^i V [i-1] \big] \,\cong\, \Motimes_{i=1}^n \Lambda_k(\wedge^i V[i-1])\,, 
\end{equation}
which is the result of Theorem~\ref{TPoly}.
We can now use the algebraic formula \eqref{HSpoly} to verify the Ault-Fiedorowicz topological formula for symmetric homology (see Conjecture~\ref{Con1} in the Introduction).
\begin{corollary} \la{hspoly}
If $k$ is a field of characteristic $0$, then there is an isomorphism
\begin{equation}
\la{Eq11}
\H_\ast\bigg(\prod_{i=1}^n \cC_{\infty}(\bS^0) \times \prod_{i=2}^n \, \Omega^{\infty}\Sigma^{\infty}(\bS^{i-1})^{\binom{n}{i}}\, ,\, k\bigg)
\,\cong\, \HS_\ast(k[x_1,\ldots,x_n])\ ,\quad \forall\,n \ge 1\,.
\end{equation}
\end{corollary}
\begin{proof} 
Since $\cC_\infty$ is an $E_\infty$-operad, and the monads associated to different $E_\infty$-operads are equivalent (see \cite{May72}), we can use the monad \eqref{DV} to compute the homology of $\cC_\infty(\bS^0)$: specifically, writing $\bS^0 = \{\ast,x\}$ as a discrete two-point space, we can identify
$$
\H_\ast(\cC_\infty(\bS^0),\,k) \,\cong\, \H_\ast[\Moplus_{n=0}^{\infty} k[E_\ast \Sigma_n^{\rm op}] \otimes_{k[\Sigma_n]} (k \cdot x)^{\otimes n}]\, \cong\, k[x]\,, 
$$
where the last isomorphism holds thanks to the fact that $k$ is a field of characteristic $0$. Then, by K\"{u}nneth's Theorem, 
\begin{equation}
\la{cso}
\H_\ast(\cC_\infty(\bS^0)\,\times\, \stackrel{n}{\ldots} \,\times\, \cC_\infty(\bS^0),\ k) \,\cong\, \Sym_k(V) \,= \,\Lambda_k(V)\,.
\end{equation}
where $V $ is the $k$-vector space spanned by the set $\{x_1, x_2, \ldots, x_n\}$
of zero-dimensional cycles corresponding to $n$ different copies of $\bS^0$.
On the other hand,
\begin{equation} 
\la{omspinf}
\H_\ast(\Omega^{\infty}\Sigma^{\infty}(\bS^{i-1}), \, k) \,\cong\, 
\H_\ast(\SP^{\infty}(\bS^{i-1}),\,k)\, \cong\, \Lambda_k(k[i-1])\,,
\end{equation}
where again the first isomorphism holds thanks
to the fact that $k$ is a field of characteristic $0$ (see, e.g., \cite[Sect. 7.3, Remark, p. 565]{CM95} and the second 
because ${\rm SP}^{\infty}(\bS^{i-1})$ is a $K(i-1,\Z)$-space. It follows from \eqref{omspinf} that there is a (non-canonical) isomorphism
\begin{equation}\la{cso1}
\H_\ast[\Omega^{\infty}\Sigma^{\infty}(\bS^{i-1})\,\times \stackrel{{\binom{n}{i}}}{\ldots}\,\times 
\, \Omega^{\infty}\Sigma^{\infty}(\bS^{i-1}),\,k]\,  \cong\, \Lambda_k(\wedge^i V[i-1])
\end{equation}
Now, combining \eqref{cso} and \eqref{cso1}, we conclude by K\"{u}nneth's Theorem: 
$$
\H_\ast\bigg(\prod_{i=1}^n \cC_{\infty}(\bS^0) \times \prod_{i=2}^n \, \Omega^{\infty}\Sigma^{\infty}(\bS^{i-1})^{\binom{n}{i}}\, ,\, k\bigg)\ \cong\ 
\Lambda_k(V) \otimes \,
\Motimes_{i=2}^n 
\Lambda_k(\wedge^i V[i-1])\ \cong\
\Motimes_{i=1}^n 
\Lambda_k(\wedge^i V[i-1])\,,
$$
Comparing with formula \eqref{HSpoly} shows that this last homology is indeed isomorphic
to the symmetric homology of the polynomial algebra in $n$ variables.
This completes the proof of the corollary.
\end{proof}
\subsection*{Acknowledgments} The work of the first author was partially supported by NSF grant DMS 1702372 and the Simons Collaboration Grant 712995. The second author was partially supported by NSF grant DMS 1702323.  


\bibliography{derivedchar_bibtex}{}
\bibliographystyle{plain}
\end{document}